\documentclass[10pt]{amsart}
\usepackage{amsmath}
\usepackage{amssymb}
\usepackage{amsthm}
\theoremstyle{plain}
\newtheorem{theo}{Theorem}
\newtheorem*{theo*}{Theorem}

\theoremstyle{definition}

\newtheorem{lem}{Lemma}[section]
\newtheorem{defi}{Definition}[section]

\newtheorem{prop}{Proposition}[section]
\newtheorem{cor}{Corollary}[section]
\newtheorem*{note}{Note}

\theoremstyle{remark}
\newtheorem*{ack}{Acknowledgements}
\newtheorem*{rem*}{Remark}
\newtheorem{rem}{Reamrk}[section]
\begin{document}

\title{On commuting  Tonelli Hamiltonians: Autonomous case}
\author{xiaojun cui \,\,\,\,\, Ji Li}
\address{Xiaojun Cui \endgraf
Department of Mathematics,\endgraf Nanjing University,\endgraf
Nanjing, 210093,\endgraf Jiangsu Province, \endgraf People's
Republic of China.} \email{xjohncui@gmail.com,xjohncui@yahoo.com}

\curraddr{Mathematisches Institut, Albert-Ludwigs-Universit\"{a}t of Freiburg, Eckerstrasse 1, 79104, Freiburg im Breisigau, Germany}
\address{Ji Li \endgraf
Department of Mathematics,\endgraf Nanjing University,\endgraf
Nanjing, 210093,\endgraf Jiangsu Province, \endgraf People's
Republic of China.} \email{adailee.hepburn@gmail.com}
\thanks{The first author is supported by National Natural Science Foundation of China
 (Grant 10801071) and  research fellowship for postdoctoral researchers
 from the Alexander von Humboldt Foundation.}

\abstract We show that the Aubry sets, the Ma\~{n}\'{e} sets, Mather's barrier functions are the same for two commuting autonomous Tonelli
Hamiltonians. We also show the quasi-linearity of $\alpha$-functions from the dynamical point of view and the existence of common $C^{1,1}$ critical subsolution for their  associated Hamilton-Jacobi equations.
\endabstract
\maketitle

\section{Introduction}
Let $M$ be a closed, connected $C^{\infty}$ Riemannian manifold. Let $TM$ and $T^{*}M$ be the tangent bundle and
cotangent bundle of $M$, respectively. In local coordinates, we may
express them as
$$TM=\big{\{}(q,\dot{q}): q \in T_qM \big{\}}$$
and
$$T^{*}M=\big{\{}(q,p): p \in T^{*}_qM\big{\}},$$
respectively. Let $pdq$ be the Liouville form.  A $C^2$ function $H:T^{*}M \rightarrow \mathbb{R}$ is
called Tonelli Hamiltonian if $H$ satisfies the following
conditions:

 $\bullet$ $H$ is  fiberwise strictly convex, i.e., the fiberwise Hessian $\frac{\partial ^2 H}{\partial
 p^2}$ is positively definite for every $(q,p) \in T^*M$.

$\bullet$ $H$ has superlinear growth, i.e., $\frac{H(q,p)}{|p|}
\rightarrow +\infty$ as $|p| \rightarrow +\infty$, where $|\cdot |$
is the norm induced by the Riemannian metric on $M$.

For a Tonelli Hamiltonian $H$, the dynamics of the Hamilton flow
$\phi^t_H$ are well understood, thanks to the celebrated Mather
theory \cite{Man},\cite{Mat1},\cite{Mat2} and its weak KAM approach
\cite{Fa3}.

 Let $\{\cdot\}$ be the Poisson bracket. Recall that two Hamiltonians
$H_1,H_2$ are commuting (in involution) if $\{H_1,H_2\}=0$.

In this paper, we restrict ourselves to the relations in Mather theory
between dynamics of two commuting Tonelli Hamiltonians.
 We show that so many things are same for two commuting Tonelli Hamiltonians. As a byproduct, we also show quasi-linearity of Mather's $\alpha$- functions \cite{Vi}, from the view point of dynamics.

For a Tonelli Hamiltonian $H$, let $L_{H}$ be the Lagrangian
associated to $H$ by Legendre transformation, i.e., $$L_H(q,\dot
q)=p \dot q-H(q,p), $$ here $p$ and $\dot q$ are related by $\dot q
=\frac{\partial H(q,p)}{\partial p}.$ Throughout this paper,
$\mathcal{L}_H$ denotes the Legendre transformation from tangent
bundle $TM$ to cotangent bundle $T^*M$, i.e.,
$$\mathcal{L}_H(\dot q)=p \iff \dot q =\frac{\partial H(q,p)}{\partial p}.$$ For each cohomology class $c \in
H^1(M,\mathbb{R})$, Mather's $\alpha$-function is defined as
follows:

$$\alpha_H(c)=-\min_{\mu}\int (L_{H}-\eta) d\mu,$$
where $\eta$ is a smooth (throughout this paper, smoothness  means that $C^r, r \geq 2$) closed 1-form on $M$ with $[\eta]=c$ (throughout this article,
$[ \cdot]$ denotes de-Rham cohomology class of a closed 1-form); the
minimum is taken over all invariant (under the Euler-Lagrange flow
$\phi^t_{L_H}$ of $L_H$) Borel probability measures. We say that an
invariant Borel probability measure $\mu$ is $c$-minimal if $\int
(L_H- \eta) d \mu=-\alpha_H(c)$, here $[\eta]=c$. Let
$\mathfrak{M}_c$ be the set of all c-minimal measures. Let Mather
set
$$\dot{M}_{H,c}=\dot{M}_{L_H,c}=\dot{M}_{c}= \text{closure of }
\big{\{}\cup_{\mu \in \mathfrak{M}_c} \text{ support of }
\mu\big{\}}.$$ Thus, $\dot{M}_c \subset TM$.  The set
$$\stackrel{\scriptscriptstyle \ast}{M}_{H,c}=
\stackrel{\scriptscriptstyle \ast}{M}_{L_H,c}=
\stackrel{\scriptscriptstyle \ast}{M}_{c}=
\mathcal{L}_H\dot{M}_{H,c}$$ is also called Mather set. Throughout this paper, let $\pi$ be the projection of $T^{*}M$ or $TM$ along the associated fibers onto $M$, according to the circumstance.  The
projection of $\stackrel{\scriptscriptstyle \ast}{M}_{H,c}$ (or
$\dot{M}_{H,c}$, equivalently) into $M$ is called projected Mather
set. We denote the projected Mather set by
$$M_{H,c}=M_{L_H,c}=M_c.$$

For any $\mathbb{R} \ni T >0$ and any closed 1-form, let
$$h^T_{H,\eta}(q_1,q_2)=
\inf_{\gamma} \int^T_0
\big{(}L_H-\eta+\alpha_H([\eta])\big{)}(\gamma(t),\dot{\gamma}(t))dt,$$
where minimum is taken over all absolutely continuous curve $\gamma:[0,T]
\rightarrow M$  with $\gamma(0)=q_1,\gamma(T)=q_2$. Let
$$h_{H,\eta}(q_1,q_2)=\lim_{T \rightarrow +\infty}h^T_{H,\eta}(q_1,q_2).$$
Note that the convergence of the limit is nontrivial, it follows
form the convergence of Lax-Oleinik semigroup in the
time-independent case \cite{Fa2}. Let
$$\rho_c(q_1,q_2)=\rho_{L_H,c}(q_1,q_2)=\rho_{H,c}(q_1,q_2)=h_{H,\eta}(q_1,q_2)+h_{H,\eta}(q_2,q_1),$$
 here $\eta$ is a smooth closed 1-form on $M$ with
$[\eta]=c.$ Now projected Aubry set $A_{H,c}=\{q \in M: \rho_{H,c}(q,q)=0\}$. Then $\rho_c$ is a pseudo-metric on $A_{H,c}$. Now define an equivalence relation $\sim_{\rho_{H,c}}$ on $A_{H,c}$ by $q_1 \sim_{\rho_{H,c}} q_2$ iff $\rho_{H,c}(q_1,q_2)=0$. Now let quotient Aubry set $(\bar{A}_{H,c}, \rho_{H,c})$ be the quotient metric space of $A_{H,c}$ under the relation $\sim_{\rho_{H,c}}$.

 We say that an absolutely continuous curve
$\gamma:\mathbb{R} \rightarrow M$ is a $c$-minimizer, if for any
interval $[a,b]$ and any absolutely continuous curve $\gamma_1:[a,b]
\rightarrow M$ such that $\gamma_1(a)=\gamma(a)$ and
$\gamma_1(b)=\gamma(b)$, we have
$$\int^b_a(L_H-\eta+\alpha_H(c))(\gamma(t),\dot{\gamma}(t))dt \leq
\int^b_a(L_H-\eta+\alpha_H(c))(\gamma_1(t),\dot{\gamma}_1(t))dt,$$
where $\eta$ is a smooth closed 1-form on $M$ such that $[\eta]=c$.
We define Ma\~{n}\'{e} set
$$\dot{N}_{H,c}=\dot{N}_{L,c}=\dot{N}_c=\cup \big{\{}(\gamma(t),
\dot{\gamma}(t)):\gamma \text{ is a c-minimizer}\big{\}}.$$ Thus
$\dot{N}_c \subset TM$.  Let
$$\stackrel{\scriptscriptstyle \ast}{N}_{H,c}=
\stackrel{\scriptscriptstyle \ast}{N}_{L_H,c}=
\stackrel{\scriptscriptstyle \ast}{N}_{c}=
\mathcal{L}_H\dot{N}_{H,c},$$
and it is also called Ma\~{n}\'{e} set. Then, $\stackrel{\scriptscriptstyle
\ast}{N}_{c} \subset T^*M$.

Let $\gamma: \mathbb{R} \rightarrow M$ be a $c$-minimizer. Let $q$
be in $\alpha$-limit set and $q^{\prime}$ be in $\omega$-limit set
of $\gamma$. If $\rho_{H,c}(q,q^{\prime})=0$, we say that $\gamma$
is a regular c-minimizer. We define Aubry set
$$\dot{A}_{H,c}=\dot{A}_{L_H,c}=\dot{A}_c=\cup \big{\{}(\gamma(t),
\dot{\gamma}(t)):\gamma \text{ is a regular c-minimizer}\big{\}}.$$
Clearly, $\dot{A}_c \subseteq \dot{N}_c$. Let
$$\stackrel{\scriptscriptstyle \ast}{A}_{H,c}=
\stackrel{\scriptscriptstyle \ast}{A}_{L_H,c}=
\stackrel{\scriptscriptstyle \ast}{A}_{c}=
\mathcal{L}_H\dot{A}_{H,c},$$
and it is also called Aubry set.

Let projected Aubry set $$A_{H,c}=A_{L_H,c}=A_c$$ and projected Ma\~{n}\'{e} set
 $$N_{H,c}=N_{L_H,c}=N_c$$ be the projections of
$$\stackrel{\scriptscriptstyle \ast}{A}_{H,c}=
\stackrel{\scriptscriptstyle \ast}{A}_{L_H,c}=
\stackrel{\scriptscriptstyle \ast}{A}_{c} \text{ and
}\stackrel{\scriptscriptstyle \ast}{N}_{H,c}=
\stackrel{\scriptscriptstyle \ast}{N}_{L_H,c}=
\stackrel{\scriptscriptstyle \ast}{N}_{c}$$ into $M$ respectively.

We have the following inclusions:
$$\stackrel{\scriptscriptstyle \ast}{M}_{H,c} \subseteq \stackrel{\scriptscriptstyle \ast}
{A}_{H,c}\subseteq \stackrel{\scriptscriptstyle \ast}{N}_{H,c}.$$

 Let $\eta$ be a smooth closed
1-form on $M$. We introduce two semigroups of nonlinear operators
$(T^-_{H,\eta,t})_{t \geq 0}$ and $(T^+_{H,\eta,t})_{t \geq 0}$
respectively. These semigroups are the so-called Lax-Oleinik
semigroups. To define them, let us fix $u \in C^0(M,\mathbb{R})$ and
$t \geq 0$. For $q \in M$, we set
$$T^-_{H,\eta,t}u(q)=\inf_{\gamma}\Big{\{}u(\gamma(0))+\int^t_0 (L_H-\eta+\alpha_H([\eta]))
(\gamma(s),\dot{\gamma}(s))ds\Big{\}},$$ where the infimum is taken
over all  absolutely continuous curve $\gamma:[0,t]\rightarrow M$
such that $\gamma(t)=q$. Also, for $q \in M$, we set
$$T^+_{H,\eta,t}u(q)=\sup_{\gamma}\Big{\{}u(\gamma(t))-\int^t_0 (L_H-\eta+\alpha_H([\eta]))
(\gamma(s),\dot{\gamma}(s))ds\Big{\}},$$ where the supremum is taken
over all  absolutely continuous curve $\gamma:[0,t]\rightarrow M$
such that $\gamma(0)=q$. A function $u$ is a forward (resp.
backward) weak KAM solution of Hamilton-Jacobi equation
$$H(q,\eta+d_qu)=\alpha_H([\eta])$$
if $T^+_{H,\eta,t}u=u$ (resp.
 $T^-_{H,\eta,t}u=u$) for any $t \geq 0$.
Let $\mathcal{S}^{+}_H({\eta})$ (resp. $\mathcal{S}^{-}_H({\eta}))$
be the set of all forward (resp. backward) weak KAM solutions of
Hamiltonian Jacobi equation
$$H(q, \eta+d_q u)=\alpha_H([\eta]),$$
where $\eta$ is a  smooth closed 1-form on $M$. By weak KAM theory \cite{Fa3}, we have
$$\lim_{t \rightarrow \infty}T^+_{H,\eta,t}u \in \mathcal{S}^{+}_H({\eta}))$$
and
$$\lim_{t \rightarrow \infty}T^-_{H,\eta,t}u \in \mathcal{S}^{-}_H({\eta}))$$
for any $u \in C^0(M, \mathbb{R})$.
Now we can state our main results as follows:

\begin{theo}
Let $H_1,H_2$ be  two Tonelli Hamiltonians $H_1,H_2$. If
$\{H_1,H_2\}=0$, then
$$T^-_{H_1,\eta,s}T^-_{H_2,\eta,r}u=T^-_{H_2,\eta,
r}T^-_{H_1,\eta,s}u, \,\,\,\,\,\,\,\,
T^+_{H_1,\eta,s}T^+_{H_2,\eta,r}u
=T^+_{H_2,\eta,r}T^+_{H_1,\eta,s}u$$ for any $u \in
C^0(M,\mathbb{R})$, any smooth closed 1-form $\eta$ on $M$ and
$\mathbb{R} \ni s,r \geq 0$.
\end{theo}
\begin{rem}
When we complete writing down this paper, we learn to know that the result
in Theorem 1 has appeared in \cite{BT}. But the proof here is a more directly variational discussion, which is very different from \cite{BT}.
\end{rem}

\begin{theo}
Let $H_1, H_2$ be two Tonelli Hamiltonians. If $\{H_1,H_2\}=0$, then
$\mathcal{S}^{+}_{H_1}({\eta})=\mathcal{S}^{+}_{H_2}({\eta})$ and
$\mathcal{S}^{-}_{H_1}({\eta})=\mathcal{S}^{-}_{H_2}({\eta})$, for
any smooth closed 1-form $\eta$.
\end{theo}

Now we recall the definitions of barrier functions \cite{Mat2}. Let
$\eta$ be a smooth closed 1-form with $[\eta]=c$, then the first
barrier function
$$B_{H,c}(q)=h_{H,\eta}(q,q);$$
 the second barrier function
$$b_{H,c}(q)=\min_{\xi,\zeta \in
A_{H,c}}\{h_{H,\eta}(\xi, q)+h_{H,\eta}(q,\zeta)
-h_{H,\eta}(\xi,\zeta)\}.$$ Clearly,  $B_{H,c}$ and $b_{H,c}$ are
independent of the choice of closed 1-form $\eta$ with $[\eta]=c$.
\begin{theo}
Let $H_1, H_2$ be two Tonelli Hamiltonians. If $\{H_1,H_2\}=0$, then
Mather's barrier functions $B_{H_1,c}(q)=B_{H_2,c}(q)$ and
$b_{H_1,c}(q)=b_{H_2,c}(q)$, for any cohomology class $c \in H^1(M,
\mathbb{R})$.
\end{theo}

\begin{theo}
Let $H_1, H_2$ be two Tonelli Hamiltonians. If $\{H_1,H_2\}=0$, then
$\stackrel{\scriptscriptstyle
\ast}{A}_{H_1,c}=\stackrel{\scriptscriptstyle \ast}{A}_{H_2,c}$
and $\stackrel{\scriptscriptstyle
\ast}{N}_{H_1,c}=\stackrel{\scriptscriptstyle \ast}{N}_{H_2,c}$
for any cohomology class in $H^1(M, \mathbb{R})$.
\end{theo}

\begin{theo}(\textbf{Quasi-linearity of $\alpha$-function:})
Let $H_1, H_2$ be two Tonelli Hamiltonians. If $\{H_1,H_2\}=0$, then
$\alpha_{H_1+H_2}(c)=\alpha_{H_1}(c)+\alpha_{H_2}(c)$, for any
cohomology class $c \in H^1(M, \mathbb{R})$.
\end{theo}
\begin{rem}
In the case that $M=\mathbb{T}^n$, the result in Theorem 5 has been
obtained by Viterbo  by his symplectic homogenization
theory \cite{Vi}. It should be mentioned that his result also covers
the case of non-Tonelli Hamiltonians, where $\alpha$ functions are
replaced by homogenizated Hamiltonians.
\end{rem}

 For any $c \in
H^1(M, \mathbb{R})$, let $$\Sigma_{H_1}(c)=\{(q,p):H_1(q,p) \leq
\alpha_{H_1}(c)\},$$ and
$$\Sigma_{H_2}(c)=\{(q,p):H_2(q,p) \leq
\alpha_{H_2}(c)\}.$$ Let $\Sigma(c)=\Sigma_{H_1}(c) \cap
\Sigma_{H_2}(c)$. Now we have
\begin{theo}
Given any smooth closed 1-form $\eta$ with $[\eta]=c$, there exists
a $C^{1,1}$ function $u$  such that $\eta+du \subset \Sigma(c)$.
\end{theo}
\begin{rem}
This theorem implies that both Hamilton-Jacobi equations
$$H_1(q, \eta+d_qu)=\alpha_{H_1}(c)$$ and $$H_2(q,
\eta+d_qu)=\alpha_{H_2}(c)$$ have a common $C^{1,1}$ subsolution,
in the case that $\{H_1,H_2\}=0$.
\end{rem}

\begin{rem}
In a preprint \cite{Cui}, the first author has extended some results of this article to the time-periodic case.
\end{rem}

\begin{note}
The first version of this paper appeared in July, 2009 and we submitted it to a journal on  August 21, 2009. As Zavidovique pointed out, which version contained a big gap, although it is very easy to fix. In November, the corrected version appeared and on November 18, 2009,  we put it on arXiv (arXiv:0911.3471). Shortly after, Zavidovique also put on (on November 19, 2009) a reprint \cite{Za}(arXiv:0911.3739), which contains similar results.

The results of this paper were also posted by the first author at a network meeting of Humboldt Foundation (November 24-26, 2009, Heidelberg), and the announcement of results was submitted to Humboldt Foundation by the first author on September 04, 2009.
\end{note}

\section{Proof of Theorem 1}

We prove the first equality, and the second equality in the theorem can be proved
similarly.

For any point $q_0 \in M$, we will prove that
$$T^-_{H_1,\eta,s}T^-_{H_2,\eta, r}u(q_0)=T^-_{H_2,\eta, r}T^-_{H_1,\eta, s}u(q_0).$$
By the definition,
\begin{eqnarray*}T^-_{H_1, \eta, s}T^-_{H_2, \eta,r}u(q_0)&=&\min_{x
\in M}(T^-_{H_2,\eta,r}u(x)+h^s_{H_1,\eta}(x,q_0))\\
&=& \min_{x, y \in
M}(u(y)+h^r_{H_2,\eta}(y,x)+h^s_{H_1,\eta}(x,q_0)).
\end{eqnarray*}
Clearly, there exist two points $x_0,y_0$ such that
$$T^-_{H_1,\eta, s}T^-_{H_2,\eta, r}u(q   _0)=u(y_0)+h^r_{H_2,\eta}(y_0,x_0)+h^s_{H_1,\eta}(x_0,q_0)).$$
We assume that $$\gamma_1:[0,r] \rightarrow M \text{ and }
\gamma_2:[r,r+s] \rightarrow M$$ are two minimizers that reach
$h^r_{H_2,\eta}(y_0,x_0),h^s_{H_1,\eta}(x_0,q_0)$ respectively.

Now we have
\begin{lem}
$\mathcal{L}_{H_2}(\dot{\gamma}_1(r))=\mathcal{L}_{H_1}(\dot{\gamma}_2(r))$.
\end{lem}
The proof of this lemma is just a standard variational discussion. Throughout this paper, we use $*$ to denote the conjunction of curves or trajectories.
\begin{proof}
Let $\Gamma(v,t)$ be an arbitrary variation of $\gamma_1 *
\gamma_2$, here $v \in (-\epsilon, \epsilon) (0 < \epsilon \in
\mathbb{R}), t \in [0,r+s]$, $\Gamma(v,0)=\gamma_1(0),
\Gamma(v,r+s)=\gamma_2(r+s)$, and
\[
\Gamma(0,t)=\left\{
\begin{array}{cc}
\gamma_1(t) & \text{ when } 0 \leq t \leq r, \\
\gamma_2(t) & \text{ when } r \leq t \leq r+s.
\end{array}
\right.
\]
Then, for any fixed $v$, we have
\begin{eqnarray*}
&&\int^r_0(L_{H_2}-\eta+\alpha_{H_2}([\eta])) \Big{(}\Gamma(v,t),
\frac{\partial \Gamma(v,t)}{\partial t}\Big{)}dt
\\
&+&\int^{r+s}_r(L_{H_1}-\eta+\alpha_{H_1}([\eta]))
\Big{(}\Gamma(v,t), \frac{\partial \Gamma(v,t)}{\partial
t}\Big{)}dt\\
& \geq & \int^r_0(L_{H_2}-\eta+\alpha_{H_2}([\eta]))(\gamma_1(t),
\dot{\gamma}_1(t))dt\\
&+&\int^{r+s}_r(L_{H_1}-\eta+\alpha_{H_1}([\eta]))(\gamma_2(t),
\dot{\gamma}_2(t))dt.
\end{eqnarray*}
Then, we have
\begin{eqnarray*}
&&\frac{d}{dv}
\mid_{v=0}\Big{(}\int^r_0(L_{H_2}-\eta+\alpha_{H_2}([\eta])+
\int^{r+s}_r(L_{H_1}-\eta+\alpha_{H_1}([\eta])\Big{)}\Big{(}\Gamma(v,t),
\frac{\partial \Gamma(v,t)}{\partial t}\Big{)}dt\\
&=&0.
\end{eqnarray*}
Thus,
\begin{eqnarray*}
&&0\\
 &=&\frac{d}{dv}
\mid_{v=0}\Big{(}\int^r_0(L_{H_2}-\eta+\alpha_{H_2}([\eta]))+
\int^{r+s}_r(L_{H_1}-\eta+\alpha_{H_1}([\eta]))\Big{)}\Big{(}\Gamma(v,t),
\frac{\partial \Gamma(v,t)}{\partial t}\Big{)}dt\\
&=&\int^r_0\Big{(}\frac{\partial
L_{H_2}(\Gamma(0,t),\frac{d}{dt}\Gamma(0,t))}{\partial
q}\frac{\partial \Gamma (v,t)} {\partial v}\mid_{v=0}+\frac{\partial
L_{H_2}(\Gamma(0,t),\frac{d}{dt}\Gamma(0,t))}{\partial \dot
q}\frac{\partial^2 \Gamma (v,t)}{\partial v \partial t}\mid_{v=0}\Big{)}dt\\
&+&\int^{r+s}_r\Big{(}\frac{\partial
L_{H_1}(\Gamma(0,t),\frac{d}{dt}\Gamma(0,t))}{\partial
q}\frac{\partial \Gamma (v,t)}{\partial v}\mid_{v=0}+\frac{\partial
L_{H_1}(\Gamma(0,t),\frac{d}{dt}\Gamma(0,t))}{\partial \dot
q}\frac{\partial^2 \Gamma (v,t)}{\partial v \partial
t}\mid_{v=0}\Big{)}dt\\
&=&\int^r_0 \frac{\partial L_{H_2}(\gamma_1(t),
\dot{\gamma}_1(t))}{\partial q}\frac{\partial \Gamma (v,t)}{\partial
v}\mid_{v=0}dt +\int^r_0 \frac{\partial L_{H_2}(\gamma_1(t),
\dot{\gamma}_1(t))}{\partial \dot{q}}d(\frac{\partial \Gamma
(v,t)}{\partial v}\mid_{v=0})\\
&+&\int^{r+s}_s \frac{\partial L_{H_1}(\gamma_2(t),
\dot{\gamma}_2(t))}{\partial q}\frac{\partial \Gamma (v,t)}{\partial
v}\mid_{v=0}dt +\int^{r+s}_r \frac{\partial L_{H_1}(\gamma_2(t),
\dot{\gamma}_2(t))}{\partial \dot{q}}d(\frac{\partial \Gamma
(v,t)}{\partial v}\mid_{v=0})\\
&=&\Big{(}\frac{\partial L_{H_2}(\gamma_1(t),
\dot{\gamma}_1(t))}{\partial \dot{q}}\frac{\partial \Gamma
(v,t)}{\partial v}\mid_{v=0}\Big{)}|^r_0 \\
&+&\int^r_0 \Big{(}\frac{\partial L_{H_2}(\gamma_1(t),
\dot{\gamma}_1(t))}{\partial q}-\frac{d}{dt} \frac{\partial
L_{H_2}(\gamma_1(t), \dot{\gamma}_1(t))}{\partial
\dot{q}}\Big{)}(\frac{\partial \Gamma (v,t)}{\partial
v}\mid_{v=0})dt\\
&+&\Big{(}\frac{\partial L_{H_1}(\gamma_2(t),
\dot{\gamma}_2(t))}{\partial \dot{q}}\frac{\partial \Gamma
(v,t)}{\partial v}\mid_{v=0}\Big{)}|^{r+s}_{r} \\
&+&\int^{r+s}_r \Big{(}\frac{\partial L_{H_1}(\gamma_2(t),
\dot{\gamma}_2(t))}{\partial q}-\frac{d}{dt} \frac{\partial
L_{H_1}(\gamma_2(t), \dot{\gamma}_2(t))}{\partial
\dot{q}}\Big{)}(\frac{\partial \Gamma (v,t)}{\partial
v}\mid_{v=0})dt\\
&=&\frac{\partial L_{H_2}(\gamma_1(r), \dot{\gamma}_1(r))}{\partial
\dot{q}}\frac{\partial \Gamma (v,r)}{\partial
v}\mid_{v=0}-\frac{\partial L_{H_1}(\gamma_2(r),
\dot{\gamma}_2(r))}{\partial \dot{q}}\frac{\partial \Gamma
(v,r)}{\partial v}\mid_{v=0},
\end{eqnarray*}
where the second equality follows from direct calculation, together
with the facts that $\eta$ is closed, and the variation is taken to
be fixed endpoints; the last equality follows from that $\gamma_1$
is a solution of Euler-Lagrange equation associated to $L_{H_2}$,
and $\gamma_2$ is a solution of Euler-Lagrange equation associated
to $L_{H_1}$.
 So, we have
 $$\frac{\partial L_{H_2}(\gamma_1(r), \dot{\gamma}_1(r))}{\partial
\dot{q}}=\frac{\partial L_{H_1}(\gamma_2(r),
\dot{\gamma}_2(r))}{\partial \dot{q}},$$ since the above formula
holds for any variation $\Gamma$. In other words,
$\mathcal{L}_{H_2}(\dot{\gamma}_1(r))=\mathcal{L}_{H_1}(\dot{\gamma}_2(r))$,
and Lemma 2.1 follows.
\end{proof}

Let
$$\mathcal{L}_{H_2}(\dot{\gamma}_1(r))=\mathcal{L}_{H_1}(\dot{\gamma}_2(r)):=p^{\diamond}.$$
Hence, if we assume that
$\mathcal{L}_{H_1}(\dot{\gamma}_2(r+s))=p_0$, then $$(y_0,
\mathcal{L}_{H_2}(\dot{\gamma}_1(0)))= \phi^{-r}_{H_2}
\phi^{-s}_{H_1}(q_0,p_0).$$ In this case,
\begin{eqnarray*}
&&T^-_{H_1,\eta, s}T^-_{H_2,\eta, r}u(q_0)\\
&=&u(y_0)+((p-\eta)dq)\big{(}\phi^t_{H_2}(x_0,p^{\diamond})|_{[-r,0]}
*\phi^t_{H_1}(x_0,p^{\diamond})|_{[0,s]}\big{)}\\
&-&sH_1(q_0,p_0)-rH_2(x_0,p^{\diamond})
+s\alpha_{H_1}([\eta])+r\alpha_{H_2}([\eta])\\
&=&u(y_0)+((p-\eta)dq)\big{(}\phi^t_{H_2}(x_0,p^{\diamond})|_{[-r,0]}
*\phi^t_{H_1}(x_0,p^{\diamond})|_{[0,s]}\big{)}\\
&-&sH_1(q_0,p_0)-rH_2(q_0,p_0)
+s\alpha_{H_1}([\eta])+r\alpha_{H_2}([\eta]),
\end{eqnarray*}
here, and in the following, $\eta$ is regarded as a smooth section
of $T^*M$ and $(p-\eta)dq$ is regarded as a smooth 1-form on $T^*M$. Note that
the first equality follows from direct calculation; the second
equality follows from that
$\phi^{-s}_{H_1}(q_0,p_0)=(x_0,p^{\diamond})$ and the fact that
$H_2$ is constant on the trajectory of $\phi^t_{H_1}$.

 Let
$\gamma_4:[s,s+r] \rightarrow M$ be the curve such that $\pi:
\phi^t_{H_2}|_{[-r,0]}(q_0,p_0)=\gamma_4$, up to a time translation.
Similarly, let $\gamma_3:[0,s] \rightarrow M$ be the curve such that
$\pi:
\phi^t_{H_1}|_{[-s,0]}(\gamma_4(s),\mathcal{L}_{H_2}(\dot{\gamma}_4(s)))=\gamma_3$,
up to a time translation. Since $\{H_1,H_2\}=0$, we have $$\pi \circ
\phi^{-s}_{H_1}\phi^{-r}_{H_2} (q_0,p_0)=y_0.$$
Hence,
\begin{eqnarray*}T^-_{H_2,\eta, r}T^-_{H_1,\eta, s}u(q_0) &\leq& u(\pi
\circ  \phi^{-s}_{H_1}\phi^{-r}_{H_2} (q_0,p_0))+ \int^s_0
(L_{H_1}-\eta+\alpha_{H_1}(c))(\gamma_3(t),\dot{\gamma}_3(t))dt
\\&+& \int^{s+r}_s
(L_{H_2}-\eta+\alpha_{H_2}(c))(\gamma_4(t),\dot{\gamma}_4(t))dt\\&=&
u(y_0)+((p-\eta)dq)\big{(}\phi^t_{H_1}(y_0,\mathcal{L}_{H_2}(\dot{\gamma}_1(0)))|_{[0,s]} *\phi^t_{H_2}(q_0,p_0)|_{[-r,0]}\big{)}\\
&-&rH_2(q_0,p_0)-sH_1(\gamma_4(s),\mathcal{L}_{H_2}(\dot{\gamma}_4(s)))+
r\alpha_{H_2}([\eta])+s\alpha_{H_1}([\eta])\\&=&
u(y_0)+((p-\eta)dq)\big{(}\phi^t_{H_2}(x_0,p^{\diamond})|_{[-r,0]}
*\phi^t_{H_1}(x_0,p^{\diamond})|_{[0,s]}\big{)}\\
&-&rH_2(q_0,p_0)-sH_1(q_0,p_0)+
r\alpha_{H_2}([\eta])+s\alpha_{H_1}([\eta])\\&=& T^-_{H_1, \eta,
s}T^-_{H_2, \eta, r}u(q_0),
\end{eqnarray*}
where the first inequality follows from the definition of
$T^-_{H,\eta,t}$,the first equality follows from the direct
calculation. For the second equality, we should say some more words:
 \begin{eqnarray*}
 &&((p-\eta)dq)\big{(}\phi^t_{H_1}(y_0,\mathcal{L}_{H_2}(\dot{\gamma}_1(0)))|_{[0,s]} *\phi^t_{H_2}(q_0,p_0)|_{[-r,0]}\big{)}\\
 &-&((p-\eta)dq)\big{(}\phi^t_{H_2}(x_0,p^{\diamond})|_{[-r,0]}
*\phi^t_{H_1}(x_0,p^{\diamond})|_{[0,s]}\big{)}\\
&=& \langle dp \wedge dq, \phi^t_{H_1}|_{[0,s]}\big{(}\phi^t_{H_2}(x_0,p^{\diamond})|_{[-r,0]}\big{)} \rangle\\
&=&0,
\end{eqnarray*}
here, the first equality follows from Stokes' formula (recall that $\{H_1,H_2\}=0$), the last equality follows from $$\phi^t_{H_1}|_{[0,s]}\big{(}\phi^t_{H_2}(x_0,p^{\diamond})|_{[-r,0]}\big{)}$$
is isotropic.

Similarly, we can prove the opposite inequality, and hence
Theorem 1 is proved.

Based on Theorem 1, we have the following propositions, which are crucial in the proof of Theorem 2.
\begin{prop}
$\mathcal{S}^{-}_{H_1,\eta} \cap \mathcal{S}^{-}_{H_2,\eta} \neq \emptyset, \mathcal{S}^{+}_{H_1,\eta} \cap \mathcal{S}^{+}_{H_2,\eta} \neq \emptyset.$
\end{prop}
\begin{proof}
We only prove the the first relation, and the second one can be proved similarly.

For any $u \in \mathcal{S}^{-}_{H_1,\eta}$ and for any $s, r \in [0,\infty)$, we have ,
$$T^-_{H_1,\eta,s}T^-_{H_2,\eta,r}u=T^-_{H_2,\eta,
r}T^-_{H_1,\eta,s}u=T^-_{H_2,\eta,
r}u.$$
Now let $s \rightarrow \infty$, we have $$\lim_{s \rightarrow \infty}T^-_{H_1,\eta,s}T^-_{H_2,\eta,r}u=\lim_{s \rightarrow \infty}T^-_{H_2,\eta,
r}u=T^-_{H_2,\eta,
r}u \in \mathcal{S}^-_{H_1, \eta}$$ for any $r \in [0,\infty)$, by weak KAM theory \cite{Fa3}.
Now let $r \rightarrow \infty$, we have $T^-_{H_2,\eta,
r}u$ converges uniformly  to a function $ u^* \in \mathcal{S}^{-}_{H_2,\eta}$ \cite{Fa3}. In fact, we also have
$u^* \in \mathcal{S}^{-}_{H_1,\eta}$. This  follows from the stability of backward weak KAM solutions \cite{Fa3}, since $T^-_{H_2,\eta, r}u \in \mathcal{S}^{-}_{H_1,\eta}$ for each $r$. Hence, we have  $u^* \in \mathcal{S}^{-}_{H_1,\eta} \cap \mathcal{S}^{-}_{H_2,\eta}$ and the first relation is proved.
\end{proof}
\begin{prop}
$H_2|_{\stackrel{\scriptscriptstyle
\ast}{A}_{H_1,[\eta]}}=\alpha_{H_2}([\eta]); H_1|_{\stackrel{\scriptscriptstyle
\ast}{A}_{H_2,[\eta]}}=\alpha_{H_1}([\eta])$.
\end{prop}
\begin{proof}
Throughout this paper, $\overline{\eta+du}$ denotes the closure of the set of $$\{(q,\eta_q+d_qu)|u
\text{ is differentiable at } q\}.$$

Choose $u^* \in \mathcal{S}^{-}_{H_1,\eta} \cap \mathcal{S}^{-}_{H_2,\eta}$, we have
$\alpha_{H_2}([\eta])=H_2|_{\overline{\eta+du^*}}$. Since $\stackrel{\scriptscriptstyle
\ast}{A}_{H_1,[\eta]} \subset \overline{\eta+du^*}$, the first equality holds.
The second equality follows similarly.
\end{proof}

\section{Proof of Theorem 2}
Let $\eta$ be any smooth closed 1-form on $M$. Now, we will prove
that if $u \in \mathcal{S}^{-}_{H_1}(\eta)$, then $u \in
\mathcal{S}^{-}_{H_2}(\eta)$.

Firstly, we will show that if $u \in \mathcal{S}^{-}_{H_1}(\eta)$,
then $H_2|_{\overline{\eta+du}}=\text{constant}$. It follows from $u \in
\mathcal{S}^{-}_{H_1}(\eta)$ that $u$ is Lipschitz, hence $u$ is
differentiable almost everywhere (with respect to Lebesgue measure).
Since $\{H_1,H_2\}=0$, $H_2$ is constant along this trajectory
of $\phi^t_{H_1}$. Let $q$ be a differentiable point of $u$, then
there exists an unique trajectory $(q(t),p(t): t \in (-\infty,0]$ of
the Hamilton flow $\phi^t_{H_1}$ such that $q(0)=q,p(0)=\eta|_q +
d_qu$ and the limit set of $\big{(}q(t),p(t): t \in
(-\infty,0]\big{)}$ lies in $\stackrel{\scriptscriptstyle
\ast}{A}_{H_1,[\eta]}$. Hence, $H_2$ is constant on the closure of
this trajectory, and  so, $H_2$ is constant on some compact subset of $\stackrel{\scriptscriptstyle
\ast}{A}_{H_1,[\eta]}$. By Proposition 2.2, we have $H_2|_{\eta+du}$, hence $H_2|_{\overline{\eta+du}}$ is constant and the constant is $\alpha_{H_2}([\eta])$.

Next, we will show that $u \in \mathcal{S}^{-}_{H_2}(\eta)$, by
showing that $u$ is a viscosity solution of
$$H_2(q,\eta+d_qu)=\alpha_{H_2}([\eta]).$$

Let us  recall  the definition of viscosity solution of
Hamiltonian-Jacobi equation. Firstly, let us fix  arbitrarily a
smooth closed 1-form $\eta$ on $M$.  A function $u:M \rightarrow
\mathbb{R}$ is a viscosity subsolution of Hamilton-Jacobi equation
$$H(q,\eta+d_qu)=d$$ if for every $C^1$ function $\phi:M\rightarrow
\mathbb{R}$ and every point $q_0 \in M$ such that $u-\phi$ has a
maximum at $q_0$, we have $H(q_0, \eta|_{q_0}+d_{q_0}\phi) \leq d$.
A function $u:M \rightarrow \mathbb{R}$ is a viscosity supersolution
of Hamilton-Jacobi equation
$$H(q,\eta +d_qu)=d$$ if for every $C^1$ function $\phi:M\rightarrow
\mathbb{R}$ and every point $q_0 \in M$ such that $u-\phi$ has a
minimum at $q_0$, we have $H(q_0,\eta|_{q_0}+d_{q_0}\phi) \geq d$.

A function $u:M \rightarrow \mathbb{R}$ is a viscosity solution of
Hamilton-Jacobi equation $$H(q,\eta+d_qu)=d$$ if it is both a
subsolution and a supersolution. The set of viscosity solutions of
$$H(q,\eta+d_qu)=d$$ is denoted by $\mathcal{S}^v_H(\eta)$.

For a Tonelli Hamilton $H$, the Hamilton-Jacobi equation
$$H(q,\eta+d_qu)=d$$
has a viscosity solution if and only if $d=\alpha_{H}([\eta])$
\cite{Fa3}. Moreover, a function $u$ is a viscosity solution of
$$H(q,\eta+d_qu)=\alpha_{H}([\eta])$$ if and only if $u \in
\mathcal{S}^-_{H}(\eta)$ \cite{Fa3}.

Now we continue the proof. Since $u$ is Lipschitz and
$H_2(q,\eta+d_qu)=\alpha_{H_2}([\eta])$ on each differentiable
point of $u$, we have $u$ is a subsolution of
$$H_2(q,\eta+d_qu)=\alpha_{H_2}([\eta])$$
\cite{Fa3}. Now we will prove that $ u$ is also a viscosity
supersolution of
$$H_2(q,\eta+d_qu)=\alpha_{H_2}([\eta]).$$

Now we need
\begin{defi}
If $u:M \rightarrow \mathbb{R}$ a function, we say that the linear
form $p \in T^*_{x_0}M$ is a lower differential of u at $x_0$ , if
we can find a neighborhood $V$ of $x_0$ and a function $\phi:V
\rightarrow \mathbb{R}$, differentiable at $x_0$, with
$\phi(x_0)=u(x_0)$ and $d_{x_0} \phi=p$, and such that $\phi(x) \leq
u(x)$ for every $x \in V$. We denote by $D_-u(x_0)$ the set of lower
differential of $u$ at $x_0$.
\end{defi}
Now we only need to show that for each $q \in M$ and $p \in
D_{-}u(q)$, we have
$$H_2(q,\eta|_q+p) \geq \alpha_{H_2}([\eta]).$$

Recall the definition of semi-concave function (with linear
modulus) \cite{CS}, \cite{Fa3}. Let us fix once and for all a finite
atlas $\Phi$ of $M$ composed of charts $\phi:B_3 \rightarrow M$,
where $B_r$ is the open ball of radius $r$ centered at zero in
$\mathbb{R}^d$. We assume that the sets $\phi(B_1) (\phi \in \Phi)$
cover $M$. A family $\mathcal{F}$ of $C^2$ functions is said
$K$-bounded if
$$|d^2(u \circ \phi)_x| \leq K$$
for all $x\in B_1, \phi \in \Phi, u \in \mathcal{F}$.
\begin{defi}
A function $u:M \rightarrow \mathbb{R}$ is called $K$-semi concave (with linear modulus)
if there exists a $K$-bounded subset $\mathcal{F}_u$ of
$C^2(M,\mathbb{R})$ such that
$$u=\inf_{f \in \mathcal{F}_u}f.$$
The constant $K$ is  also called semi-concave constant. A function
$u:M \rightarrow \mathbb{R}$ is called $K$-semi-convex if $-u$ is
$K$-semi-concave.
\end{defi}

The following lemma is due to Fathi \cite{Fa3}, and we state it here
with slight modifications:

\begin{lem}
There exists a constant $K>0$, such that $u$ is $K$-semi-concave,
for each $u \in \mathcal{S}^-_{H}(\eta)$ or $-u \in
\mathcal{S}^+_{H}(\eta)$.
\end{lem}

So, if  $u$ is a backward weak KAM solution of $H_1$,  then $u$ is
semi-concave \cite{Fa3}. Hence, $u$ is differentiable at $q$ if
$D_{-}u(q) \neq \emptyset$. So, we have that  $ u$ is also a
viscosity supersolution of
$$H_2(q,\eta+d_qu)=\alpha_{H_2}([\eta]),$$
since $H_2|_{\overline{\eta+du}}=\alpha_{H_2}([\eta])$ \cite{Fa3}. Now
we have that $u$ is a viscosity solution of
$$H_2(q,\eta+d_qu)=\alpha_{H_2}([\eta]).$$
  Hence, $u$ is also a
backward weak KAM solution to Hamilton-Jacobi equation
 $$H_2(q,\eta+d_qu)=\alpha_{H_2}([\eta]),$$
by weak KAM theory \cite{Fa3}.

Similarly, if $u \in \mathcal{S}^{-}_{H_2}(\eta)$, then $u \in
\mathcal{S}^{-}_{H_1}(\eta)$.

Thus, $\mathcal{S}^{-}_{H_1}(\eta)=\mathcal{S}^{-}_{H_2}(\eta).$

Now we will show that
$\mathcal{S}^{+}_{H_1}(\eta)=\mathcal{S}^{+}_{H_2}(\eta).$  Before
we enter into the proof, we recall the definition of symmetrical
Hamiltonian. Let $H(q,p)$ be a Tonelli Hamiltonian, then the
symmetrical Hamiltonian (with respect to $\eta$) is defined as $\check{H}(q,\eta+p)=H(q,\eta-p)$.
\begin{lem}
$L_{\check{H}}(q,\dot q)-\eta(\dot q)=L_{H}(q,-\dot q)-\eta(-\dot q).$
\end{lem}
\begin{proof}In fact, we have
\begin{eqnarray*}
&&L_{\check{H}}(q,\dot q)-\eta \dot q\\
&=&(p-\eta)\dot q-\check{H}(q,p)\\
&=&p_1 \dot q-\check{H}(q,\eta+p_1) \,\,\, ( p_1=p-\eta)\\
&=& p_1\dot q-H(q, \eta-p_1)\\
&=& (p_2+\eta)\dot q-H(q,-p_2)\,\,\, (-p_2=\eta-p_1)\\
&=& (-p_2)(-\dot q)-H(q,-p_2)-\eta(-\dot q)\\
&=&L_H(q,-\dot q)-\eta(-\dot q),
\end{eqnarray*}
where $p$ and $\dot q$ are related by $\dot q=\frac{\partial
\check{H}(q,p)}{\partial p}$ in the first equality.  All the inequalities are obviously, except the last one. We must check the last equality, since where the Legendrian equality is used. Note that the relation
can also be expressed as
$$\dot q=\frac{\partial \check H(q,p)}{\partial p}=\frac{\partial {H}(q,-p+2\eta)}{\partial
p}=-\frac{\partial {H}(q,-p+2\eta)}{\partial (-p+2\eta)}.$$
  In other words,
$$-\dot q=\frac{\partial {H}(q,-p_2)}{\partial (-p_2)}.$$
 Hence, if
$\dot q=\mathcal{L}_{\check H}(p)$, then $-\dot q=\mathcal{L}_H(-p_2)$. Thus,
the lemma follows.
\end{proof}
 The Lagrangian $\check{L}_H:=L_{\check{H}}$ is also called
symmetrical Lagrangian to $L_H$. By the fundamental result of weak
KAM theorem, we have $u \in \mathcal{S}^{+}_{H}(\eta)$ if and only
if $-u \in \mathcal{S}^{-}_{\check{H}}(\eta)$. Hence, we only need
to show that
$\mathcal{S}^{-}_{\check{H}_1}=\mathcal{S}^{-}_{\check{H}_2}$, when
$\{H_1,H_2\}=0$. Clearly, $\{\check{H}_1,\check{H}_2\}=0$, whenever
$\{H_1,H_2\}=0$. Hence, it follows from the above discussions that
$\mathcal{S}^+_{H_1}(\eta)=\mathcal{S}^+_{H_2}(\eta).$

Thus, Theorem 2 is proved.

\section{Proof of Theorem 3}
\begin{defi}
For a Tonelli Hamiltonian $H$ and a smooth closed 1-form $\eta$, we
say $u_-\in \mathcal{S}^-_{H}(\eta)$ and $u_+\in
\mathcal{S}^+_{H}(\eta)$ are conjugate with respect to $H$ if
$u_-=u_+$ on the projected Mather set $M_{H,[\eta]}$. If $u_-$ and
$u_+$ are conjugate with respect to $H$, we also denote this
relation by $u_- \sim_H u_+$.
\end{defi}

Based on this definition, we can express equivalent definitions
\cite{Fa3} of $A_{H,c}$ and $N_{H,c}$ as follows:
$$A_{H,c}=\cap \Big{\{}q:u_-(q)=u_+(q), \text{ where } u_-\in
\mathcal{S}^-_{H}(\eta), u_+\in \mathcal{S}^+_{H}(\eta), u_- \sim_H
u_+\Big{\}}$$ and
$$N_{H,c}=\cup \Big{\{}q:u_-(q)=u_+(q), \text{ where } u_-\in
\mathcal{S}^-_{H}(\eta), u_+\in \mathcal{S}^+_{H}(\eta), u_- \sim_H
u_+\Big{\}}.$$ Consequently, we also have \cite{Fa3}
$$\stackrel{\scriptscriptstyle
\ast}{A}_{H,c}=\cap\Big{\{}(q,p)|p=d_qu_-=d_qu_+:u_-\in
\mathcal{S}^-_{H}(\eta), u_+\in \mathcal{S}^+_{H}(\eta), u_- \sim_H
u_+\Big{\}}$$ and
$$\stackrel{\scriptscriptstyle
\ast}{N}_{H,c}=\cup \Big{\{}(q,p)|p=d_qu_-=d_qu_+:u_-\in
\mathcal{S}^-_{H}(\eta), u_+\in \mathcal{S}^+_{H}(\eta), u_- \sim_H
u_+\Big{\}}.$$
\begin{prop}
If $\{H_1, H_2\}=0$, then $\stackrel{\scriptscriptstyle
\ast}M_{H_1,c} \subseteq \stackrel{\scriptscriptstyle
\ast}A_{H_2,c}, \stackrel{\scriptscriptstyle \ast} M_{H_2,c}
\subseteq \stackrel{\scriptscriptstyle \ast}A_{H_1,c}$, for any
cohomology class $c \in H^1(M, \mathbb{R})$.
\end{prop}
\begin{proof}
We only need to show that $\stackrel{\scriptscriptstyle
\ast}M_{H_1,c} \subseteq \stackrel{\scriptscriptstyle \ast}A_{H_2,c}
$, by the symmetry of $H_1$ and $H_2$. By weak KAM theory
\cite{Fa3}, we have
$$\stackrel{\scriptscriptstyle
\ast}M_{H_1,c} \subseteq \cap_{u_- \in \mathcal{S}^-_{H_1}(\eta)}
\big{\{}\eta+du_-\big{\}}=\cap_{u_- \in\mathcal{S}^-_{H_2}(\eta)}
\big{\{}\eta+du_-\big{\}}.$$ By the result of Sorrentino \cite{So},
$\stackrel{\scriptscriptstyle \ast}M_{H_1,c}$ is also  invariant
under the flow of $\phi^t_{H_2}$, since $\{H_1,H_2\}=0$. Note that
$\stackrel{\scriptscriptstyle \ast}M_{H_1,c}$ lies in the graph of
$\eta+du_-$, for any $u_- \in \mathcal{S}^-_{H_1}(\eta)=
\mathcal{S}^-_{H_2}(\eta)$.

In \cite{Fa3}, Fathi proved the following lemma:
\begin{lem}
For any two points $q_0,q_1$, we have the following equality

$$h_{H,\eta}(q_0,q_1)=\sup\Big{\{}u_-(q_1)-u_+(q_0):u_-\in
\mathcal{S}^-_{H}(\eta), u_+\in \mathcal{S}^+_{H}(\eta), u_- \sim_H
u_+\Big{\}}.$$ Moreover, for any given $q_0,q_1 \in M$, this
supremum is actually attained.
\end{lem}
As a consequence of this lemma, together with the definition of
conjugate pair of weak KAM solutions,  we have
\begin{cor}
$$h_{H,\eta}(q_0,q_1)=\sup \Big{\{}u_-(q_1)-u_-(q_0):u_-\in
\mathcal{S}^-_{H}(\eta)\Big{\}}$$ for any two points $q_0,q_1 \in A_{H,c}$.
\end{cor}

We also need the definition of dominated function. Let $L_{H}$ be
the Lagrangian associated to Hamintonian $H$. Recall that a function
$f:M \rightarrow \mathbb{R}$ is dominated by $L_H$ if for each
absolutely continuous curve $\gamma:[a,b] \rightarrow M$,
we have
$$f(\gamma(b))-f(\gamma(a)) \leq \int^b_a
L_H(\gamma(t),\dot{\gamma}(t))dt.$$ If $f$ is dominated by $L_H$, we
denote it by $f \prec L_H$.

Choose any $(q_0,p_0) \in \stackrel{\scriptscriptstyle
\ast}M_{H_1,c}$, we will show that $(q_0,p_0) \in
\stackrel{\scriptscriptstyle \ast}A_{H_2,c}$. Fix a smooth closed 1-form $\eta$ with $[\eta]=c$.

Firstly, we will show that $\pi \circ \phi^t_{H_2}(q_0,p_0)$ is a
$c$-minimizer with respect to $L_{H_2}$. For any $t_1 < t_2 \in
\mathbb{R}$, and any absolutely continuous curve $\gamma_1:[t_1,t_2]
\rightarrow M$ with $$\gamma_1(t_1)=\pi \circ
\phi^{t_1}_{H_2}(q_0,p_0), \gamma_1(t_2)=\pi \circ
\phi^{t_2}_{H_2}(q_0,p_0),$$ we have
\begin{eqnarray*}
&&\int^{t_2}_{t_1}(L_{H_2}-\eta+\alpha_{H_2}([\eta]))(\pi \circ
\phi^t_{H_2}(q_0,p_0), \frac{d}{dt}({\pi \circ \phi}^t_{H_2}(q_0,p_0)))dt\\
&=&\int^{t_2}_{t_1}((p-\eta)dq) (\phi^t_{H_2}(q_0,p_0))\\
 &=&u_-(\pi \circ \phi^{t_2}_{H_2}(q_0,p_0))-u_-(\pi \circ
\phi^{t_1}_{H_2}(q_0,p_0))\\
&\leq & \int^{t_2}_{t_1}
(L_{H_2}-\eta+\alpha_{H_2}([\eta]))(\gamma_1(t),
\dot{\gamma}_1(t))dt,
\end{eqnarray*}
where $u_- \in
\mathcal{S}^-_{H_1}(\eta)(=\mathcal{S}^-_{H_2}(\eta))$; the first
equality follows from the fact that
$H_2(\phi^t_{H_2}(q_0,p_0))=\alpha_{H_2}([\eta])$; the second
equality follows from the fact that
 $\phi^t_{H_2}(q_0,p_0)$
lies in $\cap_{u_- \in\mathcal{S}^-_{H_2}(\eta)}
\big{\{}\eta+du_-\big{\}}$; the inequality follows from the fact
that $u_- \prec L_{H_2}-\eta+\alpha_{H_2}([\eta])$. Thus,
$$\stackrel{\scriptscriptstyle \ast}M_{H_1,c} \subseteq
\stackrel{\scriptscriptstyle \ast}N_{H_2,c}.$$ Hence, we only need
to show that $\rho_{H_2,c}(q_{\alpha},q_{\omega})=0$, for any
$q_{\alpha}$ lies in the $\alpha$-limit set and $q_{\omega}$ lies in
the $\omega$-limit set of $\pi \circ \phi^t_{H_2}(q_0,p_0)$.
Clearly, both  $q_{\alpha}$ and $q_{\omega}$ lie in $A_{H_2,c}$,
since $\pi \circ \phi^t_{H_2}(q_0,p_0)$ is c-minimizer with respect
to $L_{H_2}$. Thus, we can use the formula in Corollary 4.1 to
calculate $\rho_{H_2,c}(q_{\alpha},q_{\omega})$:
\begin{eqnarray*}
\rho_{H_2,c}(q_{\alpha},q_{\omega})&=&h_{H_2,\eta}(q_{\alpha},q_{\omega})+
h_{H_2,\eta}(q_{\omega},q_{\alpha})\\
&=&\sup \Big{\{}u_-(q_{\omega})-u_-(q_{\alpha}):u_-\in
\mathcal{S}^-_{H_2}(\eta)\Big{\}}\\&+&\sup
\Big{\{}v_-(q_{\alpha})-v_-(q_{\omega}):v_-\in
\mathcal{S}^-_{H_2}(\eta)\Big{\}}\\&=& \sup
\Big{\{}u_-(q_{\omega})-u_-(q_{\alpha}):u_-\in
\mathcal{S}^-_{H_1}(\eta)\Big{\}}\\&+&\sup
\Big{\{}v_-(q_{\alpha})-v_-(q_{\omega}):v_-\in
\mathcal{S}^-_{H_1}(\eta)\Big{\}}\\&=&
h_{H_1,\eta}(q_{\alpha},q_{\omega})+
h_{H_1,\eta}(q_{\omega},q_{\alpha})\\
&=&\rho_{H_1,c}(q_{\alpha},q_{\omega}),
\end{eqnarray*}
where, the second quality follows from Corollary 4.1; the third
equality follows from the fact that
$\mathcal{S}^-_{H_1}(\eta)=\mathcal{S}^-_{H_2}(\eta)$; the fourth
equality follows from Corollary 4.1, since $q_{\alpha}$ and $q_{\omega}$ also lie in $A_{H_1,c}$ by the invariance of $\stackrel{\scriptscriptstyle \ast}M_{H_1,c}$ under the flow $\phi^t_{H_2}$.

Now we claim that $\rho_{H_1,c}(q_{\alpha},q_{\omega})=0$.  Since
$q_{\alpha}$ and $q_{\omega}$ lie in the $\alpha$-limit set and
$\omega$-limit set of $\pi \circ \phi^t_{H_2}(q_0,p_0)$
respectively, there exists $t_i, t_k \rightarrow +\infty$ as $i, k
\rightarrow +\infty$, such that $$ \pi \circ
\phi^{-t_i}_{H_2}(q_0,p_0) \rightarrow q_{\alpha}, \,\,\,\,\, \pi
\circ \phi^{t_k}_{H_2}(q_0,p_0) \rightarrow q_{\omega}.$$  Now we
have $$\rho_{H_1,c}(\pi \circ \phi^{-t_i}_{H_2}(q_0,p_0),\pi \circ
\phi^{t_k}_{H_2}(q_0,p_0))=0,$$ since  $\pi \circ
\phi^{-t_i}_{H_2}(q_0,p_0)$ and $\pi \circ
\phi^{t_k}_{H_2}(q_0,p_0)$ can be connected by a $C^2$ curve $\pi
\circ \phi^t_{H_2}(q_0,p_0)$ which lies in $M_{H_1,c}$ and
$\rho_{H_1,c}$ satisfies
$$\rho_{H_1,c}(q_0,q_1) \leq Cd(q_0,q_1)^2$$ for each $q_0,q_1 \in
A_{H_1,c}$ \cite{Mat2}, where $C$ is a constant, and $d$ is the
distance induced by Riemannian metric. So,
$$\rho_{H_1,c}(q_{\alpha},q_{\omega})= 0,$$
by taking a limit.

Thus, Proposition 4.1 follows.
\end{proof}
\begin{prop}
Assume that $\{H_1, H_2\}=0$. Let $u_- \in
\mathcal{S}^-_{H_1}(\eta)=\mathcal{S}^-_{H_2}(\eta),$ $u_+ \in
\mathcal{S}^+_{H_1}(\eta)=\mathcal{S}^+_{H_2}(\eta).$ Then $u_-$ and
$u_+$ are conjugate with respect to $H_1$ if and only if $u_-$ and
$u_+$ are conjugate with respect to $H_2$, i.e., $u_- \sim_{H_1}u_+
\iff u_- \sim_{H_2}u_+.$
\end{prop}

\begin{proof}
It is a directly consequence of Proposition 4.1  and the fact
\cite{Fa3} that
$$A_{H,c}=\cap \big{\{}q: u_-(q)=u_+(q), \text{ here } u_-\in
\mathcal{S}^-_{H}(\eta), u_+\in \mathcal{S}^+_{H}(\eta), u_- \sim_H
u_+\big{\}}.$$
\end{proof}

In \cite{Fa1},\cite{Fa3}, Fathi showed that
\begin{lem}
$$B_{H,c}(q)=\sup \Big{\{}u_{-}(q)-u_{+}(q):u_-\in
\mathcal{S}^-_{H}(\eta), u_+\in \mathcal{S}^+_{H}(\eta), u_- \sim_H
u_+\Big{\}},$$ and, moreover, the supremum is attained  for each
$q$.
\end{lem}
 Recall that
\begin{eqnarray*}
b_{H,c}(q)&=&\inf_{\xi,\zeta \in A_{H,c}}\Big{\{}h_{H,\eta}(\xi,
q)+h_{H,\eta}(q,\zeta) -h_{H,\eta}(\xi,\zeta)\Big{\}}\\
&\Big{(}&=\min_{\xi,\zeta \in A_{H,c}}\Big{\{}h_{H,\eta}(\xi,
q)+h_{H,\eta}(q,\zeta) -h_{H,\eta}(\xi,\zeta)\Big{\}}\Big{)}.
\end{eqnarray*}

 In fact, the following lemma also appeared in \cite{Fa1}:

\begin{lem}
$$b_{H,c}(q)=\inf\Big{\{}u_{-}(q)-u_{+}(q):u_-\in
\mathcal{S}^-_{H}(\eta), u_+\in \mathcal{S}^+_{H}(\eta), u_- \sim_H
u_+\Big{\}}.$$ Moreover,  the infimum is attained for each $q$.
\end{lem}
Since we can not find an explicit proof of this lemma in the
literature, we will give a proof for the completeness.

\begin{proof}
For any conjugate pair $u_- \in \mathcal{S}^-_{H}(\eta), u_+ \in
\mathcal{S}^+_{H}(\eta), u_- \sim_H u_+$, we have
\begin{eqnarray*}
b_{H,c}(q)&=&\min_{\xi,\zeta}\Big{\{}h_{H,\eta}(\xi,
q)+h_{H,\eta}(q,\zeta)
-h_{H,\eta}(\xi,\zeta):\xi,\zeta \in A_{H,c}\Big{\}}\\
&\leq &\min_{\xi,\zeta}\Big{\{}h_{H,\eta}(\xi,
q)+h_{H,\eta}(q,\zeta)-(u_-(\zeta)-u_+(\xi)): \xi,\zeta \in A_{H,c} \Big{\}}\\
&=&\min_{\xi,\zeta}\Big{\{}\big{(}u_-(\xi)+h_{H,\eta}(\xi,
q)\big{)}-\big{(}u_+(\zeta)-h_{H,\eta}(q,\zeta)\big{)}: \xi,\zeta
\in
A_{H,c}\Big{\}}\\
&=& \min_{\xi \in
A_{H,c}}\Big{\{}u_-(\xi)+h_{H,\eta}(\xi,q)\Big{\}}-
\max_{\zeta \in A_{H,c}}\Big{\{}u_+(\zeta)-h_{H,\eta}(q,\zeta)\Big{\}}\\
&=&u_-(q)-u_+(q),
\end{eqnarray*}
where the  inequality follows from the fact \cite{Fa3} that
$$u_-(\zeta)-u_+(\xi) \leq h_{H,\eta}(\xi,\zeta),$$
the last equality follows from the constructions of weak KAM
solutions \cite{CIS}, \cite{Co}, \cite{Fa3}: $$u_-(q)=\min_{\xi \in
A_{H,c}}\Big{\{}u_-(\xi)+h_{H,\eta}(\xi,q)\Big{\}}$$ and
$$u_+(q)=\max_{\zeta \in A_{H,c}}\Big{\{}u_+(\zeta)-h_{H,\eta}(q,\zeta)\Big{\}}.$$
Thus, we have
$$b_{H,c}(q)\leq \inf\Big{\{}u_{-}(q)-u_{+}(q):u_-\in
\mathcal{S}^-_{H}(\eta), u_+\in \mathcal{S}^+_{H}(\eta), u_- \sim_H
u_+\Big{\}}.$$

 Next, we will show that
 $$b_{H,c}(q)= \inf\Big{\{}u_{-}(q)-u_{+}(q):u_-\in
\mathcal{S}^-_{H}(\eta), u_+\in \mathcal{S}^+_{H}(\eta), u_- \sim_H
u_+\Big{\}}.$$
 Let $\xi_0,\zeta_0 \in A_{H,c}$ such that

 $$b_c(q)=h_{H,\eta}(\xi_0, q)+h_{H,\eta}(q,\zeta_0)
-h_{H,\eta}(\xi_0,\zeta_0).$$ Then there exists a conjugate pair
$u_- \in \mathcal{S}^-_{H}(\eta), u_+ \in \mathcal{S}^+_{H}(\eta),
u_- \sim_H u_+$, such that
 \begin{eqnarray*}
 b_c(q)&=&h_{H,\eta}(\xi_0, q)+h_{H,\eta}(q,\zeta_0)
-h_{H,\eta}(\xi_0,\zeta_0)\\
 &=&\big{(}h_{H,\eta}(\xi_0,
q)+h_{H,\eta}(q,\zeta_0)\big{)}-\big{(}u_-(\zeta_0)-u_+(\xi_0)\big{)}\\
&=&\Big{(}u_-(\xi_0)+h_{H,\eta}(\xi_0,
q)\Big{)}-\Big{(}u_+(\zeta_0)-h_{H,\eta}(q, \zeta_0)\Big{)}\\
&\geq& u_-(q)-u_+(q),
 \end{eqnarray*}
where the second inequality follows from Lemma 4.1 (where the supremum is obtained); the third equality
follows from the fact that $\xi_0,\zeta_0 \in A_{H,c}$ and $u_-=u_+$ on
$A_{H,c}$ for each conjugate pair $u_-$ and $u_+$; and the
inequality follows from the constructions of weak KAM solutions.

Thus, Lemma  4.3 follows.
\end{proof}

Now Theorem 3 follows from Theorem 1, Proposition 4.2, Lemma 4.2,
and Lemma 4.3.

\begin{cor}
$b_{H,c}$ is semi-concave.
\end{cor}
\begin{proof}
By Lemma 3.1, we have $u_--u_+$ is $K$-semi-concave for each
conjugate pair
$$u_- \in \mathcal{S}^-_{H}(\eta), u_+ \in
\mathcal{S}^+_{H}(\eta), u_- \sim_H u_+.$$  It should be stressed
that $K$ is independent of the choice of conjugate pair. The
corollary follows from the fact that the infimum of a family of
semi-concave functions with the same semi-concave constant  is also
semi-concave.
\end{proof}
\begin{cor}
$c \rightarrow b_{H,c}$ is lower-semi-continuous. As a consequence,
$c \rightarrow N_{H,c}$ is upper-semi-continuous, as a set-valued
function \cite{Mat2}.
\end{cor}
\begin{proof}
It is a direct consequence of Lemma 4.3 and stability of viscosity
solutions \cite{Fa3}.
\end{proof}
Clearly, we also have
\begin{cor}
$(\bar{A}_{H_1,c},\rho_{H_1,c})$ and $(\bar{A}_{H_2,c},\rho_{H_2,c})$ are isometric, for any $c \in H^1(M,\mathbb{R})$.
\end{cor}

\section{Proof of Theorem 4}
Since
$$\stackrel{\scriptscriptstyle
\ast}{A}_{H,c}=\cap\Big{\{}(q,p):p=d_qu_-=d_qu_+,u_-\in
\mathcal{S}^-_{H}(\eta), u_+\in \mathcal{S}^+_{H}(\eta), u_- \sim_H
u_+\Big{\}}$$ and
$$\stackrel{\scriptscriptstyle
\ast}{N}_{H,c}=\cup \Big{\{}(q,p):p=d_qu_-=d_qu_+, u_-\in
\mathcal{S}^-_{H}(\eta), u_+\in \mathcal{S}^+_{H}(\eta), u_- \sim_H
u_+\Big{\}}.$$

 Theorem 4 follows from Theorem 2 and Proposition 4.2.

\section{Proof of Theorem 5}
Firstly, note that $\{H_1+H_2,H_1\}=\{H_1+H_2, H_2\}=0$, if
$\{H_1,H_2\}=0$. So, by Theorem 3, we have
$$\stackrel{\scriptscriptstyle
\ast}{A}_{H_1+H_2,c}=\stackrel{\scriptscriptstyle
\ast}{A}_{H_1,c}=\stackrel{\scriptscriptstyle \ast}{A}_{H_2,c}.$$

Now chose any point $(q_0,p_0) \in \stackrel{\scriptscriptstyle
\ast}{A}_{H_1+H_2,c}=\stackrel{\scriptscriptstyle
\ast}{A}_{H_1,c}=\stackrel{\scriptscriptstyle \ast}{A}_{H_2,c}$,
then
$$\alpha_{H_1+H_2}(c)=(H_1+H_2)(q_0,p_0)=H_1(q_0,p_0)+
H_2(q_0,p_0)=\alpha_{H_1}(c)+\alpha_{H_2}(c).$$

By the quasi-linearity of $\alpha$-function, we have
\begin{rem}
In Theorem 1, if we let $s=r$, it is easy to verify  that
$$T^-_{H_1,\eta,t}T^-_{H_2,\eta,t}u=T^-_{H_2,\eta, t}T^-_{H_1,\eta,t}u=T^-_{H_1+H_2, \eta,
t}u,$$ and $$ T^+_{H_1,\eta,t}T^+_{H_2,\eta,t}u
=T^+_{H_2,\eta,t}T^+_{H_1,\eta,t}u=T^+_{H_1+H_2, \eta, t}u$$ for any
$u \in C^0(M,\mathbb{R})$, any smooth closed 1-form $\eta$ on $M$
and $\mathbb{R} \ni t \geq 0$.
\end{rem}

\section{Proof of Theorem 6}

 Let $\eta$ be a smooth closed 1-form on $M$
with $[\eta]=c$. Choose $u \in \mathcal{S}^+_{H_1}(\eta)
(=\mathcal{S}^+_{H_2}(\eta))$, we will prove that
$T^-_{H_1,\eta,s}T^-_{H_2,\eta,r}u$ is a $C^{1,1}$ subsolution of
both Hamilton-Jacobi equations:
$$H_1(q, \eta+du)=\alpha_{H_1}(c)$$
and
$$H_2(q, \eta+du)=\alpha_{H_2}(c),$$
provided that $r,s$ are sufficiently small.

 Since $u$ is semi-convex,
so there exists $\epsilon_0
>0$ such that both  $T^-_{H_1,\eta,s}u$ and $T^-_{H_2,\eta,r}u$ are semi-convex
and semi-concave functions, hence both functions are $C^{1,1}$, in
the case that $s,r < \epsilon_0$ \cite{Be}. So, by the same
discussion, there exists $\epsilon_1
>0$ such that when $s,r <\epsilon_1$, we have $T^-_{H_1,\eta, s}T^-_{H_2,\eta, r}u$
is a $C^{1,1}$ function.

In the following, we will show that $T^-_{H_1,\eta,s}T^-_{H_2,\eta,
r}u$ is a subsolution of both Hamilton-Jacobi equations. Since
$\{H_1,H_2\}=0$, by the above lemma, we have $T^-_{H_1,\eta,
s}T^-_{H_2, \eta, r}u=T^-_{H_2,\eta, r}T^-_{H_1,\eta, s}u.$ Clearly,
$T^-_{H_2, \eta, r}u $ is a subsolution of $H_2(q,
\eta+du)=\alpha_{H_2}(c)$, and $T^-_{H_1, \eta, s}u $ is a
subsolution of $H_1(q, \eta+du)=\alpha_{H_1}(c).$

Now we need another useful lemma is proved in \cite{Fa3}:
\begin{lem}
Given a Lipschitz function $u: M \rightarrow \mathbb{R},$ the
following properties are equivalent:

$\bullet$ $ u$ is a subsolution of $H(q,\eta+du)=\alpha_H(c)$.

$\bullet$ The function $[0, +\infty) \ni t \rightarrow
T^-_{H,\eta,t}u(q)$ is non-decreasing for each $q \in M$.

$\bullet$ The function $[0, +\infty) \ni t \rightarrow
T^+_{H,\eta,t}u(q)$ is non-increasing for each $q \in M$.
\end{lem}

Now we prove that  $T^-_{H_1,\eta,s}T^-_{H_2, \eta, r}u$ is a
subsolution of
$$H_1(q, \eta+du)=\alpha_{H_1}(c).$$
Clearly, we only need to show that $T^-_{H_2,\eta,r}u$ is a
subsolution of
$$H_1(q, \eta+du)=\alpha_{H_1}(c).$$
By Lemma 7.1, we just  need to show that $[0, +\infty) \ni s
\rightarrow T^-_{H_1,\eta,s}T^-_{H_2,\eta,r}u$ is non-decreasing for
each $q \in M$ and $r >0$. Since
$T^-_{H_1,\eta,s}T^-_{H_2,\eta,r}u=T^-_{H_2,\eta,r}T^-_{H_1,\eta,s}u,$
it follows from the following two facts:

1. $[0, +\infty) \ni s \rightarrow T^-_{H_1,\eta,s}u$ is
non-decreasing, since $u$ is a subsolution of
$$H_1(q, \eta+du)=\alpha_{H_1}(c);$$

2. $T^-_{H_2,\eta,r }$ has the monotony property, i.e., for each
$u,v \in C^0(M,\mathbb{R})$ and all $r >0$, we have
$$u \leq v \Rightarrow T^-_{H_2,\eta,r} u \leq T^-_{H_2,\eta,r} v.$$

Similarly, we have that $T^-_{H_1,\eta,s}T^-_{H_2,\eta,r}u$ is also
a subsolution of
$$H_2(q, \eta+du)=\alpha_{H_2}(c).$$

Theorem 6 follows.

\begin{ack}
We would like to thank Professor C.-Q. Cheng for encouragements and
helps for many years. This paper was completed when the first author
visited Albert-Ludwigs-Universit\"{a}t of Freiburg as a postdoctoral researcher, supported
by a fellowship from the Alexander Von Humboldt Foundation. The
first author would like to thank Professor V. Bangert and Mathematisches Institut at Albert-Ludwigs-Universit\"{a}t of Freiburg for their hospitality.
\end{ack}

\end{document}